\documentclass[11pt]{article}

\usepackage[a4paper,margin=1in]{geometry}
\usepackage{amsmath,amssymb,amsthm,mathtools}
\usepackage{enumitem}
\usepackage{xcolor}
\usepackage{hyperref}
\usepackage{listings}

\hypersetup{
  colorlinks=true,
  linkcolor=blue,
  citecolor=blue,
  urlcolor=blue
}

\lstdefinestyle{sagestyle}{
  language=Python,
  basicstyle=\ttfamily\scriptsize,
  breaklines=true,
  frame=single,
  numbers=left,
  numberstyle=\tiny,
  columns=fullflexible,
  showstringspaces=false
}

\newtheorem{theorem}{Theorem}[section]
\newtheorem{proposition}[theorem]{Proposition}
\newtheorem{lemma}[theorem]{Lemma}
\newtheorem{corollary}[theorem]{Corollary}
\theoremstyle{definition}

\theoremstyle{remark}

\newcommand{\F}{\mathbb F}
\newcommand{\Fq}{\mathbb F_q}

\title{Consecutive non-square non-primitive tuples in finite fields}
\author{Juncheng Zhou and Hongfeng Wu\footnote{Corresponding author.}
\setcounter{footnote}{-1}
\footnote{E-mail addresses:
jczhoumath@gmail.com (J. Zhou), whfmath@gmail.com (H. Wu)}
\\
{College of Science, North China University of Technology, Beijing, China}\\
}
\date{}

\begin{document}

\maketitle

\begin{abstract}
Let $q$ be an odd prime power and put
\[
  \theta_q=\frac{\varphi(q-1)}{q-1}.
\]
An element of $\Fq$ is called non-square non-primitive, or \emph{NSNP},
if it is both a non-square and a non-primitive element.
Let $\ell$ be an odd prime divisor of $q-1$.  We obtain a character-sum
estimate for the number of translates of an arbitrary finite set lying
in the set of non-square $\ell$th powers.  Combining this estimate with
a finite computation, we prove that $\theta_q<4/15$ guarantees three
consecutive NSNP elements.  On the boundary $\theta_q=4/15$, the only
exceptions are $q\in\{31,61,121\}$.  As a further application, we show
that for $\operatorname{char}\Fq>3$, the inequality $\theta_q<8/35$
guarantees four consecutive NSNP elements, with $q=211$ as the unique
exception on the boundary
$\theta_q=8/35$.  Thus both constants are best possible.
\end{abstract}



\section{Introduction}

Let $\Fq$ be the finite field of odd order $q$.  Write
\[
  \theta_q=\frac{\varphi(q-1)}{q-1},
\]
so that $\theta_q$ is the proportion of primitive elements in
$\Fq^*$.  An element of $\Fq$ will be called \emph{NSNP} if it is
both a non-square and a non-primitive element.

The parameter $\theta_q$ has a direct density interpretation.  Since
$\mathbb F_q^\times$ is cyclic of even order $q-1$, every primitive
element is a non-square.  Indeed, if $g$ is a generator of
$\mathbb F_q^\times$, then every primitive element has the form $g^m$
with $\gcd(m,q-1)=1$; since $q-1$ is even, the exponent $m$ must be
odd, whereas the squares are precisely the even powers of $g$.
Consequently,
\[
 \#\{a\in\mathbb F_q^\times:a\text{ is NSNP}\}
 =
 \frac{q-1}{2}-\varphi(q-1)
 =
 \left(\frac12-\theta_q\right)(q-1).
\]
Thus $\frac12-\theta_q$ is exactly the density of NSNP elements in
$\mathbb F_q^\times$.  Moreover, if $\ell$ is an odd prime divisor of
$q-1$, then every $\ell$th power in $\mathbb F_q^\times$ is
non-primitive.  Hence the non-square $\ell$th powers form a structured
subset of the NSNP elements that can be studied by multiplicative
characters.  In our argument, $\ell$ is chosen to be the least odd
prime divisor of $q-1$.  The sharp constants obtained below are also
naturally reflected in the Euler products
\[
 \frac4{15}
 =
 \left(1-\frac12\right)
 \left(1-\frac13\right)
 \left(1-\frac15\right),
 \qquad
 \frac8{35}
 =
 \left(1-\frac12\right)
 \left(1-\frac13\right)
 \left(1-\frac15\right)
 \left(1-\frac17\right).
\]
They correspond respectively to the prime supports
$\{2,3,5\}$ and $\{2,3,5,7\}$ of $q-1$; the boundary arguments in
Sections~3 and~4 make this correspondence precise.

The existence of consecutive primitive elements, and of consecutive
elements satisfying related multiplicative restrictions, has been
studied extensively; see, for example,
\cite{Cohen1985,CohenSilvaTrudgian2015,Cohen2024}.  In~\cite{Cohen2026},
Cohen proved that if $\theta_q<1/3$, then $\Fq$ contains a pair of
consecutive NSNP elements, and determined the exceptional fields on the
boundary $\theta_q=1/3$.  He also asked whether there is an absolute
constant $c>0$ such that $\theta_q<c$ guarantees three consecutive
NSNP elements.

We answer Cohen's question with the following sharp result.

\begin{theorem}\label{thm:main-triple}
Let $q$ be an odd prime power.  If
\[
\theta_q<\frac{4}{15},
\]
then $\Fq$ contains three consecutive NSNP elements.

If
\[
\theta_q=\frac{4}{15},
\]
then the same conclusion holds except when
\[
q\in\{31,61,121\}.
\]
In particular, the constant $4/15$ is sharp.
\end{theorem}

The sharpness follows since the three exceptional values above all satisfy
$\theta_q=4/15$, while none of the corresponding fields contains an NSNP
triple.

The main ingredient is an estimate for the number of translates of an
arbitrary finite set.  Let $\ell$ be an odd prime divisor of $q-1$,
put $L=2\ell$, and let
$\mathcal H\subseteq\Fq$ have $k$ elements.  If
$M_{\mathcal H,\ell}(q)$ counts the translates $x+\mathcal H$ lying
in the set of non-square $\ell$th powers, then
\[
 \left|M_{\mathcal H,\ell}(q)-\frac{q}{L^k}\right|
 \le \frac{A_k(L)}{L^k}\sqrt q+\frac{k}{L}.
\]
where $A_k(L)=((k-1)L-k)L^{k-1}+1$.
In particular, when $\mathcal H=\{0,1,\ldots,k-1\}$,
$M_{\mathcal H,\ell}(q)$ counts consecutive tuples.

As a further application of the same method, we obtain a sharp result
for 4-tuples.

\begin{theorem}\label{thm:main-quadruple}
Let $q$ be a prime power with $\operatorname{char}\Fq>3$.  If
\[
  \theta_q<\frac{8}{35},
\]
then $\Fq$ contains four consecutive NSNP elements.

If
\[
  \theta_q=\frac{8}{35},
\]
then the same conclusion holds except when
\[
q=211.
\]
In particular, the constant $8/35$ is sharp.
\end{theorem}

Section~\ref{sec:k-tuples} proves the counting estimate.
Section~\ref{sec:triples} proves Theorem~\ref{thm:main-triple}, and
Section~\ref{sec:quadruples} proves Theorem~\ref{thm:main-quadruple}
by the same method.  The verification code is included in
Appendix~\ref{app:code}.

\section{Counting translates in non-square
\texorpdfstring{$\ell$}{ell}th powers}\label{sec:k-tuples}

Let $\ell$ be an odd prime divisor of $q-1$, put $L=2\ell$, and let
$\mathcal N_\ell$ denote the set of non-square $\ell$th powers in
$\Fq$.  Every element of $\mathcal N_\ell$ is non-primitive.

Let
\[
  \mathcal H=\{h_0,\ldots,h_{k-1}\}\subseteq\Fq
\]
be a set of $k\ge2$ distinct elements, and define
\[
  M_{\mathcal H,\ell}(q)
  =\#\{x\in\Fq:x+h_s\in\mathcal N_\ell
                  \text{ for }0\le s\le k-1\}.
\]
When $\operatorname{char}\Fq\ge k$ and
$\mathcal H=\{0,1,\ldots,k-1\}$, we write this number as
$M_{k,\ell}(q)$.

Let $\chi$ be a multiplicative character of exact order $L$.  For
every $0\le i\le L-1$, extend $\chi^i$ to zero by
$\chi^i(0)=0$; in particular, $\chi^0$ is the characteristic
function of $\Fq^\times$.  The indicator of $\mathcal N_\ell$ is
\[
\begin{aligned}
  \mathbf 1_{\mathcal N_\ell}(x)
  &=\frac{1}{2\ell}(1-\chi^\ell(x))
       \sum_{i=0}^{\ell-1}\chi^{2i}(x)\\
  &=\frac1L\sum_{i=0}^{L-1}(-1)^i\chi^i(x).
\end{aligned}
\]
Here $\chi^\ell$ is the quadratic character, while
$\ell^{-1}\sum_{i=0}^{\ell-1}\chi^{2i}$ is the indicator of the
nonzero $\ell$th powers.  Since $\ell$ is odd, multiplication by
$\chi^\ell$ shifts the even exponents modulo $2\ell$ onto the odd
exponents.  Our convention at zero makes the identity valid on all
of $\Fq$.

We use the following Weil bound; see~\cite{Schmidt}.

\begin{lemma}[Weil bound]\label{lem:weil}
Let $\chi$ be a nontrivial multiplicative character of order $d$
over $\mathbb F_q$, and let $f(x)\in\mathbb F_q[x]$.  Suppose that
$f$ has exactly $m>1$ distinct zeros in its splitting field and is
not of the form
\[
f(x)=c\,g(x)^d
\]
for some $c\in\overline{\mathbb F}_q^\times$ and
$g(x)\in\overline{\mathbb F}_q[x]$.  Then
\[
\left|
\sum_{x\in\mathbb F_q}\chi(f(x))
\right|
\le (m-1)\sqrt q.
\]
\end{lemma}

Substituting the characteristic function of $\mathcal N_\ell$ gives
\[
L^k M_{\mathcal{H},\ell}(q)
=
\sum_{\boldsymbol{i}\in\{0,\ldots,L-1\}^k}
(-1)^{i_0+\cdots+i_{k-1}}
\sum_{x\in\mathbb F_q}
\prod_{s=0}^{k-1}\chi^{i_s}(x+h_s),
\]
where
\[
\boldsymbol{i}=(i_0,\ldots,i_{k-1}).
\]

The term corresponding to $\boldsymbol{i}=\boldsymbol{0}$ is
\[
\sum_{x\in\mathbb F_q}
\prod_{s=0}^{k-1}\chi^0(x+h_s)
=q-k,
\]
since $\chi^0$ is the characteristic function of
$\mathbb F_q^\times$ and the $k$ elements
$-h_0,\ldots,-h_{k-1}$ are distinct.

Now fix $\boldsymbol{i}\ne\boldsymbol{0}$, and put
\[
I(\boldsymbol{i})
=
\{s\in\{0,\ldots,k-1\}:i_s\ne0\},
\qquad
r(\boldsymbol{i})=|I(\boldsymbol{i})|.
\]
Define
\[
f_{\boldsymbol{i}}(X)
=
\prod_{s\in I(\boldsymbol{i})}(X+h_s)^{i_s}.
\]
Since $\chi^0=\mathbf 1_{\mathbb F_q^\times}$, we have
\[
\prod_{s=0}^{k-1}\chi^{i_s}(x+h_s)
=
\chi\bigl(f_{\boldsymbol{i}}(x)\bigr)
\prod_{s\notin I(\boldsymbol{i})}\chi^0(x+h_s).
\]
Consequently,
\[
\sum_{x\in\mathbb F_q}
\prod_{s=0}^{k-1}\chi^{i_s}(x+h_s)
=
\sum_{\substack{x\in\mathbb F_q\\
x\ne-h_s\ {\rm for}\ s\notin I(\boldsymbol{i})}}
\chi\bigl(f_{\boldsymbol{i}}(x)\bigr).
\]

The polynomial $f_{\boldsymbol{i}}$ has precisely
$r(\boldsymbol{i})$ distinct zeros and is not a constant multiple of
an $L$th power.  If $r(\boldsymbol{i})\ge2$, Lemma~\ref{lem:weil}
therefore gives
\[
\left|
\sum_{x\in\mathbb F_q}
\chi\bigl(f_{\boldsymbol{i}}(x)\bigr)
\right|
\le
\bigl(r(\boldsymbol{i})-1\bigr)\sqrt q.
\]
The same estimate also holds when $r(\boldsymbol{i})=1$, since in
that case the complete character sum is zero.

Passing from the complete sum to the preceding restricted sum removes
at most $k-r(\boldsymbol{i})$ terms, each of absolute value at most
$1$.  Hence
\[
\left|
\sum_{x\in\mathbb F_q}
\prod_{s=0}^{k-1}\chi^{i_s}(x+h_s)
\right|
\le
\bigl(r(\boldsymbol{i})-1\bigr)\sqrt q
+k-r(\boldsymbol{i}).
\]

For each $1\le r\le k$, there are
\[
\binom{k}{r}(L-1)^r
\]
index vectors having support of size $r$.  It follows that
\[
\begin{aligned}
\left|L^kM_{\mathcal{H},\ell}(q)-(q-k)\right|
&\le
\sum_{r=1}^k
\binom{k}{r}(L-1)^r
\bigl((r-1)\sqrt q+k-r\bigr)\\
&=
A_k(L)\sqrt q+B_k(L),
\end{aligned}
\]
where
\[
A_k(L)
=
\sum_{r=1}^k\binom{k}{r}(L-1)^r(r-1)
=
\bigl((k-1)L-k\bigr)L^{k-1}+1,
\]
and
\[
B_k(L)
=
\sum_{r=1}^k\binom{k}{r}(L-1)^r(k-r)
=
k(L^{k-1}-1).
\]
Since $B_k(L)+k=kL^{k-1}$, we obtain
\[
\left|L^kM_{\mathcal H,\ell}(q)-q\right|
\leq
A_k(L)\sqrt q+kL^{k-1}.
\]
Dividing by $L^k$ gives the following theorem.

\begin{theorem}
\label{thm:k-tuple-estimate}
Let $\ell$ be an odd prime divisor of $q-1$, put $L=2\ell$, and let
$\mathcal H\subseteq\Fq$ consist of $k\ge2$ distinct elements.  With
$A_k(L)$ as defined above,
\begin{equation}\label{eq:refined-pattern-bound}
 \left|M_{\mathcal H,\ell}(q)-\frac{q}{L^k}\right|
 \le\frac{A_k(L)}{L^k}\sqrt q+\frac{k}{L}.
\end{equation}
\end{theorem}

\begin{corollary}\label{cor:pattern-existence}
Under the hypotheses of Theorem~\ref{thm:k-tuple-estimate}, let
$A_k(L)$ be as in that theorem.  If
\begin{equation}\label{eq:simple-existence-condition}
 q>\bigl(A_k(L)+1\bigr)^2,
\end{equation}
then some translate $x+\mathcal H$ is contained in
$\mathcal N_\ell$.
\end{corollary}

\begin{proof}
Put $u=\sqrt q$.  Theorem~\ref{thm:k-tuple-estimate} gives
\[
L^kM_{\mathcal H,\ell}(q)
\ge u^2-A_k(L)u-kL^{k-1}.
\]
Its discriminant is
\[
\Delta=A_k(L)^2+4kL^{k-1},
\]
and its positive root is
\[
\frac{A_k(L)+\sqrt\Delta}{2}.
\]
Rationalizing the numerator gives
\[
\frac{A_k(L)+\sqrt\Delta}{2}
=A_k(L)+
\frac{2kL^{k-1}}{A_k(L)+\sqrt\Delta}
<A_k(L)+\frac{kL^{k-1}}{A_k(L)}.
\]
Moreover, since $L=2\ell\ge6$ and $k\ge2$,
\[
\frac{kL^{k-1}}{A_k(L)}
=\frac{kL^{k-1}}
{\bigl((k-1)L-k\bigr)L^{k-1}+1}
<\frac{k}{(k-1)L-k}
=\frac{1}{(k-1)L/k-1}
<1.
\]
Thus the positive root is less than $A_k(L)+1$.  The hypothesis
$q>(A_k(L)+1)^2$ therefore makes the quadratic positive, and hence
$M_{\mathcal H,\ell}(q)>0$.
\end{proof}

\section{The optimal threshold for NSNP triples}\label{sec:triples}

Suppose that $\theta_q\le4/15$.  Then $q-1$ is not a power of $2$,
since otherwise $\theta_q=1/2$.  Let $\ell$ be the least odd prime
divisor of $q-1$ and put $L=2\ell$.

Taking $k=3$ and $\mathcal H=\{0,1,2\}$ in
Corollary~\ref{cor:pattern-existence} gives three consecutive
non-square $\ell$th powers whenever
\begin{equation}\label{eq:triple-large-q-condition}
 q>(2L^3-3L^2+2)^2,
\end{equation}
In particular, the right-hand side of
\eqref{eq:triple-large-q-condition} is $106276$ for $\ell=3$ and
$2896804$ for $\ell=5$.

\begin{lemma}\label{lem:triple-reduction}
Suppose that $\theta_q<4/15$ and that
\eqref{eq:triple-large-q-condition} fails.  Then
\[
 \ell=3\qquad\text{and}\qquad q\le106276.
\]
\end{lemma}

\begin{proof}
Suppose first that $\ell\ge7$.  If $q-1$ had at most ten distinct
odd prime divisors, then
\[
 \theta_q\ge
 \frac12\prod_{7\le p\le41}\left(1-\frac1p\right)
 >\frac4{15},
\]
a contradiction.  Thus $q-1$ has at least eleven distinct odd prime
divisors, and hence
\[
 q>2\ell^{11}>(2L^3-3L^2+2)^2.
\]
This contradicts the failure of
\eqref{eq:triple-large-q-condition}.  Therefore
$\ell\in\{3,5\}$.

Now suppose that $\ell=5$.  The inequality
\[
 \frac12\prod_{5\le p\le17}\left(1-\frac1p\right)
 >\frac4{15}
\]
shows that $q-1$ has at least six distinct odd prime divisors.  Hence
\[
 q\ge1+2\cdot5\cdot7\cdot11\cdot13\cdot17\cdot19
 =3233231.
\]
Since $3233231>2896804$, this is again a contradiction.  Thus
$\ell=3$, and the failure of~\eqref{eq:triple-large-q-condition}
gives $q\le106276$.
\end{proof}

\begin{proposition}\label{prop:finite-check}
Let $\ell$ be the least odd prime divisor of $q-1$.
Among the odd prime powers satisfying
\[
 \theta_q<\frac4{15},\qquad \ell=3,
 \qquad q\le106276,
\]
there are exactly $1374$ candidates, and all contain an NSNP triple.
\end{proposition}

\begin{proof}
This follows from the SageMath computation in
Appendix~\ref{app:code}.
\end{proof}

\begin{lemma}\label{lem:support-injectivity}
The map
\[
 \mathcal P\longmapsto
 \prod_{p\in\mathcal P}\left(1-\frac1p\right)
\]
is injective on finite sets of primes.
\end{lemma}

\begin{proof}
Suppose that two distinct finite sets of primes $S$ and $T$ give the
same product.
After cancelling the common factors, put
\[
A=S\setminus T,\qquad B=T\setminus S.
\]
Then $A$ and $B$ are disjoint and
\[
\prod_{p\in A}\frac{p-1}{p}
 =
\prod_{p\in B}\frac{p-1}{p}.
\]
Let $r$ be the largest prime in $A\cup B$, and suppose without
loss of generality that $r\in A$.  Cross-multiplication gives
\[
\prod_{p\in A}(p-1)\prod_{p\in B}p
 =
\prod_{p\in B}(p-1)\prod_{p\in A}p.
\]
The right-hand side is divisible by $r$, whereas the left-hand
side is not, since all of its factors are strictly smaller than
$r$.  This is a contradiction.
\end{proof}

\begin{corollary}\label{cor:boundary-primes-triple}
If $\theta_q=4/15$, then the set of prime divisors of $q-1$ is
precisely $\{2,3,5\}$.  In particular, the least odd prime divisor
of $q-1$ is $3$.
\end{corollary}

\begin{proof}
We have
\[
\theta_q
 =\prod_{p\mid q-1}\left(1-\frac1p\right)
 =\frac4{15}
 =\left(1-\frac12\right)
  \left(1-\frac13\right)
  \left(1-\frac15\right).
\]
The injectivity just proved therefore implies
\[
\{p:p\mid q-1\}=\{2,3,5\}.
\]
\end{proof}

\begin{proposition}\label{prop:boundary}
Among the odd prime powers satisfying
\[
\theta_q=\frac{4}{15}
\qquad\text{and}\qquad
q\le106276,
\]
there are exactly $56$ candidates.  Every corresponding finite
field contains an NSNP triple except for
\[
q\in\{31,61,121\}.
\]
None of these three exceptional fields contains an NSNP triple.
\end{proposition}

\begin{proof}
This follows from the SageMath computation in
Appendix~\ref{app:code}.
\end{proof}

\begin{proof}[Proof of Theorem~\ref{thm:main-triple}]
Suppose first that
\[
\theta_q<\frac4{15}.
\]
If~\eqref{eq:triple-large-q-condition} holds,
Corollary~\ref{cor:pattern-existence} gives three consecutive
non-square $\ell$th powers.  Otherwise,
Lemma~\ref{lem:triple-reduction} gives $\ell=3$ and $q\le106276$,
and Proposition~\ref{prop:finite-check} completes the proof.

Now suppose that
\[
\theta_q=\frac4{15}.
\]
By Corollary~\ref{cor:boundary-primes-triple}, we have $\ell=3$.
If~\eqref{eq:triple-large-q-condition} holds,
Corollary~\ref{cor:pattern-existence} gives an NSNP triple.  Otherwise
$q\le106276$, and
Proposition~\ref{prop:boundary} shows that an NSNP triple exists unless
\[
q\in\{31,61,121\}.
\]

Finally, for each $q\in\{31,61,121\}$, we have $\theta_q=4/15$, but
$\Fq$ contains no NSNP triple.  Consequently no constant larger than
$4/15$ can replace $4/15$ in the strict inequality.
\end{proof}

\section{The optimal threshold for NSNP 4-tuples}
\label{sec:quadruples}

Suppose that $\theta_q\le8/35$.  Then $q-1$ is not a power of $2$,
since otherwise $\theta_q=1/2$.  Let $\ell$ be the least odd prime
divisor of $q-1$ and put $L=2\ell$.

Taking $k=4$ and $\mathcal H=\{0,1,2,3\}$ in
Corollary~\ref{cor:pattern-existence} gives four consecutive
non-square $\ell$th powers whenever
\begin{equation}\label{eq:quadruple-large-q-condition}
 q>(3L^4-4L^3+2)^2,
\end{equation}
When $\ell=3$, the right-hand side is $9156676$.

\begin{lemma}\label{lem:quadruple-reduction}
Suppose that $\theta_q<8/35$ and that
\eqref{eq:quadruple-large-q-condition} fails.  Then
\[
 \ell=3\qquad\text{and}\qquad q\le9156676.
\]
\end{lemma}

\begin{proof}
As in Lemma~\ref{lem:triple-reduction}, suppose that $\ell\ge5$.
If $q-1$ had at most nine distinct odd
prime divisors, then
\[
 \theta_q\ge
 \frac12\prod_{5\le p\le31}
 \left(1-\frac1p\right)
 >\frac8{35},
\]
a contradiction.  Thus $q-1$ has at least ten distinct odd prime
divisors.  If $p_1<\cdots<p_{10}$ are the ten smallest of them, then
$p_j\ge\ell+2(j-1)$.  Since
\[
(\ell+12)(\ell+18)>48\ell,
\qquad
(\ell+14)(\ell+16)>48\ell,
\]
we obtain
\[
\begin{aligned}
 q-1
 &\ge2\prod_{j=0}^9(\ell+2j)\\
 &>\ell^6(\ell+12)(\ell+14)(\ell+16)(\ell+18)\\
 &>\ell^6(48\ell)^2=(3L^4)^2.
\end{aligned}
\]
so $\sqrt q>3L^4>3L^4-4L^3+2$.  This contradicts the failure of
\eqref{eq:quadruple-large-q-condition}.  Therefore $\ell=3$, and the
failure of~\eqref{eq:quadruple-large-q-condition} gives
$q\le9156676$.
\end{proof}

\begin{proposition}\label{prop:quadruple-strict-check}
Among the odd prime powers satisfying
\[
 \theta_q<\frac8{35},\qquad
 \ell=3,\qquad
 q\le9156676,
\]
there are exactly $13943$ candidates.  Every corresponding field
contains four consecutive NSNP elements.
\end{proposition}

\begin{proof}
This follows from the SageMath computation in
Appendix~\ref{app:code}.
\end{proof}

\begin{corollary}\label{cor:boundary-primes-quadruple}
If $\theta_q=8/35$, then the set of prime divisors of $q-1$ is
precisely $\{2,3,5,7\}$.  In particular, the least odd prime divisor
of $q-1$ is $3$.
\end{corollary}

\begin{proof}
Apply Lemma~\ref{lem:support-injectivity} to
\[
 \frac8{35}
 =\left(1-\frac12\right)
  \left(1-\frac13\right)
  \left(1-\frac15\right)
  \left(1-\frac17\right).
\]
\end{proof}

\begin{proposition}\label{prop:quadruple-boundary-check}
Among the odd prime powers satisfying
\[
 \theta_q=\frac8{35}
 \qquad\text{and}\qquad q\le9156676,
\]
there are exactly $173$ candidates.  Every corresponding field
contains four consecutive NSNP elements except $\F_{211}$, which
contains no NSNP 4-tuple.
\end{proposition}

\begin{proof}
This follows from the SageMath computation in
Appendix~\ref{app:code}.
\end{proof}

\begin{proof}[Proof of Theorem~\ref{thm:main-quadruple}]
Suppose first that $\theta_q<8/35$.  If
\eqref{eq:quadruple-large-q-condition} holds,
Corollary~\ref{cor:pattern-existence} gives four consecutive
non-square $\ell$th powers.  Otherwise,
Lemma~\ref{lem:quadruple-reduction} gives $\ell=3$ and
$q\le9156676$, and Proposition~\ref{prop:quadruple-strict-check}
completes the proof.

Now let $\theta_q=8/35$.  By
Corollary~\ref{cor:boundary-primes-quadruple}, we have $\ell=3$.  If
\eqref{eq:quadruple-large-q-condition} holds,
Corollary~\ref{cor:pattern-existence} gives the required 4-tuple.
Otherwise $q\le9156676$, and
Proposition~\ref{prop:quadruple-boundary-check} gives the conclusion
except for $q=211$.

Finally,
\[
 \theta_{211}=\frac{\varphi(210)}{210}=\frac8{35},
\]
and Proposition~\ref{prop:quadruple-boundary-check} shows that
$\F_{211}$ has no NSNP 4-tuple.  Hence no constant larger than
$8/35$ can replace $8/35$ in the strict inequality.
\end{proof}

\appendix
\refstepcounter{section}
\section*{Appendix~\thesection. SageMath verification}\label{app:code}

The following program, implemented in SageMath
9.3~\cite{SageMath}, performs the finite verifications used in
Propositions~\ref{prop:finite-check}, \ref{prop:boundary},
\ref{prop:quadruple-strict-check}, and
\ref{prop:quadruple-boundary-check}.  It
\begin{enumerate}
\item enumerates all odd prime powers in the relevant finite ranges;
\item computes $\theta_q$ and the least odd prime divisor of $q-1$;
\item tests directly whether an NSNP triple or 4-tuple exists and
records a witness whenever one does.
\end{enumerate}
Running the program gives the following results.  All $1374$ triple
candidates with $\theta_q<4/15$ contain an NSNP triple, while among
the $56$ candidates with $\theta_q=4/15$ the only exceptions are
$q\in\{31,61,121\}$.  Similarly, all $13943$ 4-tuple candidates with
$\theta_q<8/35$ contain an NSNP 4-tuple, while among the $173$
candidates with $\theta_q=8/35$ the only exception is $q=211$.

\begin{lstlisting}[style=sagestyle]
# verify_nsnp.sage
# SageMath 9.3+ compatible
#
# Exact finite verification for the triple and 4-tuple results.

from sage.all import *
from functools import lru_cache
import csv
import sys


T3 = QQ(4) / 15
T4 = QQ(8) / 35

# Analytic cutoffs from Corollary 2.3.  The finite checks use q <= B.
B3 = ZZ(106276)
B4 = ZZ(9156676)

SUPPORT3 = frozenset([ZZ(2), ZZ(3), ZZ(5)])
SUPPORT4 = frozenset([ZZ(2), ZZ(3), ZZ(5), ZZ(7)])

EXPECTED_STRICT3 = 1374
EXPECTED_BOUNDARY3 = 56
EXPECTED_STRICT4 = 13943
EXPECTED_STRICT4_PRIME = 13862
EXPECTED_STRICT4_EXTENSION = 81
EXPECTED_BOUNDARY4 = 173
EXPECTED_BOUNDARY4_PRIME = 168
EXPECTED_BOUNDARY4_EXTENSION = 5
EXPECTED_BAD3 = [ZZ(31), ZZ(61), ZZ(121)]
EXPECTED_BAD4 = [ZZ(211)]

OUTPUT_FILE = "verify_nsnp_tuples.csv"


@lru_cache(maxsize=None)
def q_minus_one_data(q):
    """Return factorization, prime support, theta_q, and least odd prime divisor."""
    q = ZZ(q)
    fac = factor(q - 1)
    support = frozenset(ZZ(r) for r, _ in fac)
    theta = prod(QQ(r - 1) / r for r in support)
    odd_primes = sorted(r for r in support if r != 2)
    ell = min(odd_primes) if odd_primes else None
    return fac, support, theta, ell


def odd_prime_powers_upto(bound):
    """Yield all (q,p,e) with q=p^e<=bound and p odd prime."""
    bound = ZZ(bound)
    for p in prime_range(3, bound + 1):
        p = ZZ(p)
        q = p
        e = 1
        while q <= bound:
            yield q, p, e
            q *= p
            e += 1


def generate_candidates():
    """Enumerate the four candidate groups in one pass."""
    strict3, boundary3 = [], []
    strict4, boundary4 = [], []

    print("Enumerating odd prime powers up to", B4)

    for q, p, e in odd_prime_powers_upto(B4):
        _, support, theta, ell = q_minus_one_data(q)
        item = (q, p, e)

        if q <= B3:
            if theta < T3 and ell == 3:
                strict3.append(item)
            if theta == T3:
                if support != SUPPORT3:
                    raise RuntimeError(
                        "Unexpected 4/15 support for q={}: {}".format(
                            q, support))
                boundary3.append(item)

        if p >= 5 and theta < T4 and ell == 3:
            strict4.append(item)
        if theta == T4:
            if support != SUPPORT4:
                raise RuntimeError(
                    "Unexpected 8/35 support for q={}: {}".format(
                        q, support))
            boundary4.append(item)

    groups = [strict3, boundary3, strict4, boundary4]
    for group in groups:
        group.sort(key=lambda item: item[0])
    return tuple(groups)


def is_nsnp_prime(a, p, odd_support):
    """Test NSNP membership in the prime field F_p."""
    a = ZZ(a) % p
    if a == 0:
        return False
    n = p - 1
    if power_mod(a, n // 2, p) != p - 1:
        return False
    return any(power_mod(a, n // r, p) == 1
               for r in odd_support)


def prime_field_witness(p, length):
    """Find a consecutive NSNP tuple in F_p, or return None."""
    _, support, _, _ = q_minus_one_data(p)
    odd_support = sorted(r for r in support if r != 2)
    for x in range(int(p)):
        if all(is_nsnp_prime(x + s, p, odd_support)
               for s in range(length)):
            return tuple(ZZ((x + s) % p) for s in range(length))
    return None


def extension_field_witness(q, length):
    """Find a consecutive NSNP tuple in a proper extension field."""
    F = GF(q, name="a")
    n = q - 1
    _, support, _, _ = q_minus_one_data(q)
    odd_support = sorted(r for r in support if r != 2)
    minus_one = -F.one()

    def is_nsnp(z):
        if z == 0 or z**(n // 2) != minus_one:
            return False
        return any(z**(n // r) == 1 for r in odd_support)

    for x in F:
        if all(is_nsnp(x + F(s)) for s in range(length)):
            return tuple(x + F(s) for s in range(length))
    return None


@lru_cache(maxsize=None)
def nsnp_witness(q, p, e, length):
    """Dispatch to the prime-field or extension-field search."""
    if e == 1:
        return prime_field_witness(p, length)
    return extension_field_witness(q, length)


def verify_group(candidates, label, length, progress_step=1000):
    """Verify one candidate group and return records and exceptions."""
    records = []
    exceptions = []
    total = len(candidates)
    print("\nChecking {} candidates: {}".format(label, total))

    for index, (q, p, e) in enumerate(candidates, start=1):
        fac, _, theta, ell = q_minus_one_data(q)
        witness = nsnp_witness(q, p, e, length)
        if witness is None:
            exceptions.append(q)
            print(" exception q={}".format(q))

        records.append({
            "group": label,
            "q": q,
            "p": p,
            "e": e,
            "factorization": fac,
            "theta": theta,
            "least_odd": ell,
            "length": length,
            "witness": witness,
        })

        if progress_step and (index % progress_step == 0 or index == total):
            print(" checked {}/{}".format(index, total))
            sys.stdout.flush()

    return records, exceptions


def run_data_for_211():
    """Return the NSNP count and all linear NSNP runs in F_211."""
    p = ZZ(211)
    _, support, _, _ = q_minus_one_data(p)
    odd_support = sorted(r for r in support if r != 2)
    membership = [is_nsnp_prime(x, p, odd_support)
                  for x in range(int(p))]
    runs = []
    x = 0
    while x < p:
        if not membership[x]:
            x += 1
            continue
        start = x
        while x < p and membership[x]:
            x += 1
        runs.append(tuple(range(start, x)))
    return sum(ZZ(v) for v in membership), runs


def write_results(filename, records):
    """Write all verification records to one CSV file."""
    fields = ["group", "q", "p", "e", "factorization_of_q_minus_1",
              "theta_q", "least_odd", "length", "witness"]
    with open(filename, "w", newline="") as stream:
        writer = csv.writer(stream)
        writer.writerow(fields)
        for item in records:
            writer.writerow([
                item["group"], item["q"], item["p"], item["e"],
                str(item["factorization"]), str(item["theta"]),
                item["least_odd"], item["length"],
                str(item["witness"]),
            ])


def main():
    strict3, boundary3, strict4, boundary4 = generate_candidates()

    assert len(strict3) == EXPECTED_STRICT3
    assert len(boundary3) == EXPECTED_BOUNDARY3
    assert len(strict4) == EXPECTED_STRICT4
    assert sum(e == 1 for _, _, e in strict4) == EXPECTED_STRICT4_PRIME
    assert sum(e > 1 for _, _, e in strict4) == EXPECTED_STRICT4_EXTENSION
    assert len(boundary4) == EXPECTED_BOUNDARY4
    assert sum(e == 1 for _, _, e in boundary4) == EXPECTED_BOUNDARY4_PRIME
    assert sum(e > 1 for _, _, e in boundary4) == EXPECTED_BOUNDARY4_EXTENSION

    rec_s3, bad_s3 = verify_group(strict3, "strict-4/15", 3)
    rec_b3, bad_b3 = verify_group(
        boundary3, "boundary-4/15", 3, progress_step=20)
    rec_s4, bad_s4 = verify_group(strict4, "strict-8/35", 4)
    rec_b4, bad_b4 = verify_group(
        boundary4, "boundary-8/35", 4, progress_step=20)

    assert bad_s3 == []
    assert bad_b3 == EXPECTED_BAD3
    assert bad_s4 == []
    assert bad_b4 == EXPECTED_BAD4

    count_211, runs_211 = run_data_for_211()
    assert count_211 == 57
    assert max(len(run) for run in runs_211) == 3
    assert [run for run in runs_211 if len(run) == 3] == [
        (26, 27, 28), (31, 32, 33), (88, 89, 90)
    ]

    records = rec_s3 + rec_b3 + rec_s4 + rec_b4
    write_results(OUTPUT_FILE, records)

    print("\n==============================================")
    print("COMBINED FINITE VERIFICATION SUMMARY")
    print("==============================================")
    print("strict triple candidates:", len(strict3))
    print("triple boundary candidates:", len(boundary3))
    print("triple boundary exceptions:", bad_b3)
    print("strict 4-tuple candidates:", len(strict4))
    print("  prime fields:", EXPECTED_STRICT4_PRIME)
    print("  proper prime powers:", EXPECTED_STRICT4_EXTENSION)
    print("4-tuple boundary candidates:", len(boundary4))
    print("4-tuple boundary exceptions:", bad_b4)
    print("NSNP elements in F_211:", count_211)
    print("maximum NSNP run in F_211:",
          max(len(run) for run in runs_211))
    print("certificate file:", OUTPUT_FILE)


main()
\end{lstlisting}

%\section*{Acknowledgment}
%The authors are grateful to the editor and the anonymous referees for taking time to read and comment on the article very %carefully and insightfully. Their nice comments and suggestions helped us to improve the article very much.

\section*{Data availability}
Data sharing not applicable to this article as no datasets were generated or analysed during the current study.

\section*{Declaration of competing interest}
The authors declare that we have no known competing financial interests or personal relationships 
that could be perceived to influence the work reported in this paper.

\end{document}